\documentclass[12pt]{amsart}
\usepackage{latexsym}
\usepackage{amsmath,amsfonts,amssymb,amsthm}
\usepackage{graphicx}
\usepackage{color}
\usepackage{xypic}
\usepackage{marvosym}
\input xy
\xyoption{all}
\newcommand{\cO}{{\mathcal O}}

\renewcommand{\cD}{{\mathcal D}}

\newcommand{\eff}{\Eff}
\newcommand{\nef}{\Nef}
\DeclareMathOperator{\mov}{mov}
\DeclareMathOperator{\fix}{fix}

\newcommand{\bQ}{\mathbb Q}
\newcommand{\bZ}{\mathbb Z}
\newcommand{\bC}{\mathbb C}

\newcommand{\bP}{\mathbb P}

\newcommand{\toric}[1]{\T\V({#1})}  %
\DeclareMathOperator{\tail}{tail}
\DeclareMathOperator{\head}{head}

\newcommand{\kG}{\Gamma} %

\DeclareMathOperator{\faces}{faces}
\newcommand{\kk}{\bC} %
\def\pDiv{\mathcal{D}} %
\DeclareMathOperator{\CDiv}{CaDiv}
\DeclareMathOperator{\TT}{eq}
\newcommand{\finer}{\leq}
\newcommand{\PD}{\Delta}
\DeclareMathOperator{\pwl}{min}
\newcommand{\altcone}{\beta} %
\DeclareMathOperator{\stable}{st}
\DeclareMathOperator{\rank}{rk}
\newcommand{\kyz}{\hspace{-0.2em}}
\DeclareMathOperator{\cp}{comp}

\definecolor{tomato}{rgb}{1.0,.388,.278}      %
\definecolor{DarkSalmon}{rgb}{.914,.588,.478} %
\definecolor{SaddleBrown}{rgb}{0.545,.271,.075}      %

\newcommand{\PP}{\mathbb P}
\newcommand{\Q}{\mathbb Q}

\newcommand{\T}{\mathbb T}
\newcommand{\V}{\mathbb V}
\newcommand{\Z}{\mathbb Z}

\newcommand{\CD}{{\mathcal D}}

\newcommand{\CF}{{\mathcal F}}

\newcommand{\CO}{{\mathcal O}}

\DeclareMathOperator{\Spec}{Spec}

\newcommand{\conv}{\operatorname{conv}} %
\DeclareMathOperator{\innt}{int}

\newcommand{\disjointcup}{\bigsqcup}
\renewcommand{\div}{{\rm div}}
\DeclareMathOperator{\Eff}{Eff}

\DeclareMathOperator{\Nef}{Nef}

\DeclareMathOperator{\Mov}{Mov}

\DeclareMathOperator{\Cox}{Cox}
\DeclareMathOperator{\Pic}{Pic}
\DeclareMathOperator{\PDiv}{PDiv}
\DeclareMathOperator{\CaDiv}{CaDiv}
\DeclareMathOperator{\Div}{Div}
\DeclareMathOperator{\Cl}{Cl}
\DeclareMathOperator{\cl}{cl}

\DeclareMathOperator{\orb}{orb}

\DeclareMathOperator{\id}{id}

\DeclareMathOperator{\mult}{mult}
\newcommand{\dual}{^{\scriptscriptstyle\vee}}
\DeclareMathOperator{\Pol}{Pol}
\DeclareMathOperator{\Hom}{Hom}

\DeclareMathOperator{\gHom}{Hom}

\newcommand{\til}[1]{{\widetilde{#1}}}
\newcommand{\surj}{\rightarrow\hspace{-0.8em}\rightarrow}

\newcommand{\ko}{\overline}
\newcommand{\ku}{\underline}

\newcommand{\kss}{\scriptscriptstyle}

\newcommand{\kbb}{{\kss \bullet}}
\renewcommand{\iff}{\Leftrightarrow}

\newcommand{\kst}{\,|\;}
\newcommand{\kSt}{\,\big|\;}

\newtheorem{theorem}{Theorem}

\newtheorem{lemma}[theorem]{Lemma}

\newtheorem{corollary}[theorem]{Corollary}

\theoremstyle{definition}
\newtheorem{definition}[theorem]{Definition}

\theoremstyle{remark}
\newtheorem*{remark}{Remark}

\numberwithin{equation}{section}
\definecolor{cpriv}{rgb}{0.4,0.1,0.2}

\definecolor{ckq}{rgb}{0.9,0.15,0.03}

\definecolor{ckn}{rgb}{0.1,0.55,0.03}

\definecolor{cJarek}{rgb}{0.1,0.45,0.03}

\newcommand{\kkk}[1]{}

\begin{document}
\title[P-Divisors of Cox Rings]{
P-Divisors of Cox Rings
}
\author[K.~Altmann]{Klaus Altmann%
}
\address{Institut f\"ur Mathematik,
FU Berlin,
Arnimalle 3, 
14195 Berlin, 
Germany}
\email{Altmann@math.fu-berlin.de
}
\author[J.~A.~Wi{\'s}niewski]{Jaros{\l}aw A.~Wi\'sniewski}%
\address{Instytut Matematyki UW, Banacha 2, 02-097 Warszawa, Poland}
\email{J.Wisniewski@mimuw.edu.pl} 

\thanks{This project was conceived when the second author visited Freie
  Universit\"at in Berlin supported by the Alexander von Humboldt
  Foundation and completed when both authors visited the Mathematical
  Sciences Research Institute in Berkeley. The second author was also
  supported by Polish MNiSzW grant N N201 2653 33. We thank all
  supporting institutions.}

 \keywords{Cox ring, torus action}

\date{}
\maketitle
\begin{abstract}
  The Cox ring of a so-called Mori Dream Space (MDS) is finitely
  generated and it is graded over the divisor class group. Hence the
  spectrum of the Cox ring comes with an action of an algebraic torus
  whose GIT quotient is the variety in question. We present the
  associated description of this Cox ring as a polyhedral divisor in
  the sense of \cite{toral}.  
Via the shape of its polyhedral coefficients,
it connects the equivariant structure of the Cox ring
with the world of stable loci and stable multiplicities of linear systems.
  \vspace{2ex}
\end{abstract}
\section{Introduction}
\label{introC2P}
Let $Z$ be a $\bQ$-factorial projective variety defined over the
field of complex numbers such that its divisor class group $\Cl(Z)$
is a lattice that is a free abelian, finitely generated group. We
consider the Cox ring of $Z$
$$\Cox(Z)=\bigoplus_{D\in\Cl(Z)}\kG(Z,\cO(D))$$
with multiplicative structure defined by a choice of divisors
whose classes form a basis of $\Cl(Z)$. Our standing assumption in
this paper is the finite generation of the $\bC$-algebra $\Cox(Z)$. We
will call such $Z$ a Mori Dream Space (or MDS) as it was baptized by
Hu and Keel in \cite{MDS}.  We note that a somewhat more general
definition of MDS, without the $\bQ$ factoriality assumption, was
developed by Artebani, Hausen and Laface, see
\cite[Thm.~2.3]{K3}. However, $\bQ$-factoriality of $Z$ is a part of
our set up in the present paper.

The $\Cl(Z)$-grading of $\Cox(Z)$ yields an algebraic action of the
associated torus $\Hom_\bZ(\Cl(Z),\bC^*)\cong
(\bC^*)^{\rank(\Cl(Z))}$ on the affine variety $\Spec(\Cox(Z))$. The
variety $Z$ is a GIT quotient of $\Spec(\Cox(Z))$ by the action of
this torus. More precisely, a choice of an ample divisor on $Z$
determines an open subset of $\Spec(\Cox((Z))$ such that $Z$ is a good
geometric quotient of this set, see \cite[Prop.~2.9]{MDS}.

Affine varieties with an algebraic torus action were dealt with by
Altmann and Hausen, \cite{toral}, who introduced the notion of
polyhedral divisors, or p-divisors. Every normal, affine variety 
$X$ with an
algebraic torus action can be described in terms of a polyhedral
divisor $\pDiv=\sum_i\Delta_i\otimes D_i$
over its Chow quotient $Y$, \cite[Thm.~3,4]{toral}. 
Alternatively, such a p-divisor can be interpreted as a convex,
fanwise linear (i.e.\ piecewise linear and homogeneous, defined on a
cone) map from the character lattice $M$ of the torus to
$\CaDiv_\Q(Y)$. See (\ref{pDivDef}) for more details.  Note that, by
abuse of notation, we use the word ``Chow quotient'' for the
normalization of the distinguished component of the inverse limit of
the GIT quotients of $X$, cf.~\cite[Sect.~6]{toral}, \cite{Hu}.

We apply this formalism to treat the case of $X=\Spec(\Cox(Z))$ for
$Z$ as above. Although, in general, the structure of the Chow quotient
$Y$ is rather obscure, our main result, Theorem~\ref{th-mainRes},
asserts that the associated p-divisor is supported on a finite number
of exceptional divisors $D_i$ with polyhedral coefficients $\Delta_i$
described
clearly in terms of stabilized multiplicities
with respect to these divisors:
$$
\;\Delta_i=\{C\in\Cl^*(Z)_\Q\kst C\geq -\mult^{\stable}_{D_i}\}
+ \mbox{\rm shift}.
$$
Thus, polyhedral divisors provide an alternative view of the
stabilized base point loci and the asymptotic order of the vanishing of
linear series on $Z$, as defined by Ein, Lazarsfeld, Musta{\c t}{\v a}, 
Nakamaye and Popa, \cite{asymptotic}.

The composition of the p-divisor associated to $\Cox(Z)$,
  treated as a fanwise linear map
  $\,\pDiv:M_\bQ=\Cl_\bQ(Z)\supset\Eff(Z)\to\CaDiv_\Q(Y)$, with the
  divisor class map $\CaDiv_\bQ(Y)\to \Pic_\bQ(Y)$ (dividing by
  $\bQ$-principal divisors), maps the cone of effective divisors on
  $Z$, denoted by $\Eff(Z)$, to the cone $\Nef(Y)$ of nef (in this
  case also semiample) divisors on $Y$. 
In Corollary~\ref{Cox-pwl-map} we show that it is a composition of two
other maps
  $\Cl_\bQ(Z)\supset\Eff(Z)\to\Cl_\Q(Z)=\Pic_\Q(Z)\to\Pic_\Q(Y)$.
First, one performs a retraction of $\Eff(Z)$ to the cone of movable
divisors $\Mov(Z)$ which is a union of cones $\Nef(Z_i)$ for $Z_i$
being 
different GIT quotients of $\Cox(Z)$. Second, the chambers $\Nef(Z_i)$
are mapped to faces of $\Nef(Y)$ by
pulling the divisors back along the natural morphisms $Y\rightarrow Z_i$.

Our starting point, however, is the toric case where both the Chow
quotient of $\Cox(Z)$ and the p-divisor can be described 
explicitly. We discuss this in Section \ref{tordown} right after the
introductory Section~\ref{pDiv} where we recall the language of
p-divisors. The main toric result, Theorem \ref{pToriCoxTheorem}, is
obtained by explicit methods. In the subsequent Section~\ref{dualpol},
we rephrase it by using dual polyhedra and the
associated fanwise linear functions. These easy observations
lead us to the relation to
multiplicities of divisors in base point loci of linear systems
forming the core of the proof of Theorem \ref{th-mainRes}.
This is contained
in Section \ref{MDS}, where we also recall the basic information about
MDS. 

Finally, in Section \ref{surf}, we discuss the 
surface case and provide some further examples.
If $Z=S$ with $\dim S=2$, then
the Chow limit $Y$ coincides with $S$. So the
p-divisor defines a retraction $\Eff(S)$ to $\Nef(S)$ 
reflecting the Zariski decomposition on $S$.
It is linear on the Zariski chambers, as defined in \cite{bks}. 
The coefficients of the p-divisor on an MDS surface are presented in 
Theorem~\ref{pToriCoxTheoremSurf}.
For a del Pezzo surface $S$ they look particularly nice:

{\bf Corollary~\ref{pToriCoxCor}.}
{\em
If $S$ is a del Pezzo with $E_i\subseteq S$ denoting their exceptional curves,
then the p-divisor encoding $\Cox(S)$ equals
$\;\pDiv=\id_{\Cl(S)}+\sum_i \big(\ko{0E_i}+\Nef(S)\big)\otimes E_i$.
}

\bigskip
\section{The language of p-divisors}\label{pDiv}%

\subsection{Definition of p-divisors}\label{pDivDef}
We start recalling the basic notions of \cite{toral}.  Let $T$ be an
affine torus over a field of complex numbers $\kk$.  It gives rise to
the mutually dual free abelian groups, or lattices,
$M:=\gHom_{\mbox{\tiny algGrp}}(T,\kk^\ast)$ and
$N:=\gHom_{\mbox{\tiny algGrp}}(\kk^\ast, T)$. The pairing of dual
lattices (or, also, dual vector spaces) will be denoted by $\langle\ \
, \ \rangle$.  Via $T=\Spec\kk[M]=N\otimes_\Z\kk^*$, the torus can be
recovered from these lattices.  Denote by $M_\Q:= M\otimes_\Z\Q$ and
$N_\Q:= N\otimes_\Z\Q$ the corresponding vector spaces over $\Q$ (the
same notation will be used whenever we extend a lattice to a
$\Q$-vector space).

\begin{definition}
\label{def-pol}
If $\sigma\subseteq N_\Q$ is a polyhedral cone, then we denote by
$\Pol(N_\Q,\sigma)$ the Grothendieck group of the semigroup
$$ 
\Pol^+(N_\Q,\sigma):=\{\Delta\subseteq N_\Q\kst
\Delta = \sigma + [\mbox{compact polytope}]\}
$$
with respect to Minkowski addition. Via $a\mapsto a+\sigma$, the latter
contains $N_\Q$.
Moreover,
$\tail(\Delta):=\sigma$ is called the tail cone of the
elements of $\Pol(N_\Q,\sigma)$.
\vspace{-2ex}
\end{definition}

Let $Y$ be a normal and semiprojective (i.e.\ $Y\to Y_0$ is projective
over an affine $Y_0$) $\kk$-variety. By $\CaDiv(Y)$ and $\Div(Y)$ we
denote the group of Cartier and Weil divisors on $Y$ with
linear equivalence groups by $\Pic(Y)$ and $\Cl(Y)$,
respectively.  A $\Q$-Cartier divisor on $Y$ is called {\em
  semiample} if a multiple of it becomes base point free.

\begin{definition}
\label{def-pp}
An element $\pDiv=\sum_i \Delta_i\otimes D_i\in 
\Pol(N_\Q,\sigma)\otimes_\Z \CaDiv(Y)$ with effective divisors $D_i$
and $\Delta_i\in\Pol^+(N_\Q,\sigma)$
is called a {\em polyhedral divisor} on $(Y,N)$ with tail cone $\sigma$.
Moreover, it is called {\em semiample} if
the evaluations
$\pDiv(u):=\sum_i \min\langle \Delta_i,u\rangle D_i$
are semiample for $u\in\sigma^\vee\cap M$
and big for $u\in\innt\sigma^\vee\cap M$.

\vspace{-2ex}
\end{definition}
Note that the membership $u\in\sigma^\vee:=\{u\in
  M_\Q\kst \langle \sigma, u\rangle \geq 0\}$ guarantees that
  $\min\langle \Delta_i,u\rangle > -\infty$ and therefore $\pDiv$
  defines a function $\sigma^\vee\to\CaDiv_\Q(Y)$ which we will denote
  by the same name. Sometimes, by abuse, we will refer to $\pDiv$ as
  a function defined on the whole lattice $M$ or space $M_\Q$. In such
  a case, for $u\not\in\sigma^\vee$, we have $\min\langle
  \Delta_i,u\rangle = -\infty$, and thus, although $-\infty$ as a
  Cartier divisor coefficient does not make sense, we get as a reasonable
  conclusion that $\kG(Y,\CO_Y(\pDiv(u)))=0$.

The common tail cone $\sigma$ of the coefficients $\Delta_i$
will be denoted by
$\tail(\pDiv)$.
Semiample polyhedral divisors will be called {\em p-divisors} for short.
Their positivity assumptions imply that
$\pDiv(u)+\pDiv(u')\leq \pDiv(u+u')$; hence
$\CO_Y(\pDiv):= \bigoplus_{u\in\sigma^\vee\cap M} \CO_Y(\pDiv(u))$
becomes a sheaf of rings,
and we can define $X:=X(\pDiv):= \Spec \kG(Y,\CO(\pDiv))$ over $Y_0$.
\\[1ex]
This space does not change if $\pDiv$ is pulled back via a
birational modification $Y'\to Y$ or if $\pDiv$ is altered by a
polyhedral {\em principal} divisor -- the latter means an image under
$N\otimes_\Z\kk(Y)^*\to \Pol(N_\Q,\sigma)\otimes_\Z \CaDiv(Y)$.
P-divisors that differ by (chains of)
those operations only are called equivalent.
Note that this implies that one can always ask for a smooth $Y$.

\begin{theorem}[\cite{toral}, Theorems (3.1), (3.4); Corollary (8.12)] 
The map $\pDiv\mapsto X(\pDiv)$ yields a bijection between
equivalence classes of p-divisors and normal, affine $\kk$-varieties with an
effective $T$-action.
\end{theorem}

\begin{remark}
The $T$-action on $X$ corresponds to the $M$-valued grading of 
$\kG(Y,\CO(\pDiv))$. In this context, $\tail(\pDiv)^\vee$ becomes the cone
generated by the weights.
Note also that the knowledge of 
$\pDiv\in \Pol(N_\Q,\sigma)\otimes_\Z \CaDiv(Y)$ is equivalent
to the knowledge of $\pDiv$ as 
the above
fanwise linear (cf.\ (\ref{pwlinear})) function
$\sigma\dual\cap M\to \CaDiv(Y)$, $u\mapsto \pDiv(u)$.
\end{remark}

\subsection{Morphisms between p-divisors}\label{pDivFunct}
The construction of $X(\pDiv)$ is functorial:
Up to the above mentioned equivalences of p-divisors,
a map $(Y',N',\pDiv')\to(Y,N,\pDiv)$ consists of 
a morphism
$\psi:Y' \to Y$ such that the support $\bigcup_i D_i$ of $\pDiv$
does not contain $\psi(Y')$ and
a linear map $F: N'\to N$  with
$$ 
\sum_i \big(F(\Delta'_i)+\tail\pDiv\big) \otimes D'_i 
=: F_\ast(\pDiv') \;\subseteq\;
\psi^\ast(\pDiv):=
\sum_i \Delta_i\otimes \psi^\ast(D_i)
\vspace{-1ex}
$$
inside $\Pol(N_\Q,\tail\pDiv)\otimes_\Z\CDiv(Y')$.
The inclusion is understood as a relation between the coefficients
of the same divisors. In particular, we ask for
$F(\tail\pDiv')\subseteq \tail\pDiv$.

\begin{theorem}[\cite{toral}, Corollary (8.14)]
A map $(Y',N',\pDiv')\to(Y,N,\pDiv)$ with dominant $\psi:Y' \to Y$
gives rise to an
equivariant, dominant map
$X(\pDiv')\to X(\pDiv)$, and, eventually,
this leads to an equivalence of categories.
\vspace{-2ex}
\end{theorem}

\subsection{p-divisors encode toric degenerations}\label{pDivBed}
The representation or encoding 
of a multigraded algebra as a p-divisor has many
advantages. First, while one misses direct information about 
generators and syzygies,
one should notice that this construction, however, entails
being of finite type. This is based on the fact that only semiample divisors
are used to produce the homogeneous parts of the algebra.
\\[1ex]
However, the main advantage of a p-divisor is that it is possible to read off
equivariant and geometric properties of the associated
affine $T$-variety $X$. This becomes possible because $X$ is the contraction
of $\til{X}:=\til{X}(\pDiv):= \Spec_Y \CO(\pDiv)$, and this space
is a degenerate toric fibration over $Y$. 
That is, there is a flat map $\til{X}\to Y$ with the
general fiber being the toric variety
$\toric{\tail(\pDiv),N}:=\Spec\kk[\tail(\pDiv)^\vee\cap M]$.
Moreover,
the divisors $D_i$ and their
polyhedral coefficients $\Delta_i$ provide the information
about the location and the quality of the degeneration, respectively:
$$ 
\xymatrix@R3ex@C5em{
\til{X} \ar[r]^{}  \ar[d]_{} &
X  \ar[d]^{} \\
Y \ar[r]^{} &
Y_0
}
$$
Special fibers over $y\in Y$ can be reducible; their components are
in a one-to-one correspondence with the vertices of
the polyhedron
$\Delta_y:=\sum_{D_i\ni y} \Delta_i$.
\\[1ex]
Thus, also the configuration of $T$-orbits and their closures is 
directly encoded
in the presentation of $X$ as a polyhedral divisor
$\pDiv$.
The orbits in $\til{X}$ correspond to pairs $(y,F)$ with $y\in Y$
and faces $F\leq \Delta_y$.
Moreover, as is it is known from the toric case,
mutual inclusions among orbit closures correspond to opposite inclusions
of the corresponding faces.
The orbit structure of $X$ may be obtained from that of $\til{X}$ by
keeping track of when certain orbits from $\til{X}$ will be identified in $X$.
This happens in relation to the different contractions of $Y$ provided by the
semiample divisors $\CD(u)$.
\\[1ex]
As an example of how to use this information, see 
in \cite{Hausen} Hausen's description of those
open subsets $U\subseteq X$ providing a complete quotient $U/T$.

\section{The toric situation}\label{tordown}

\subsection{Restriction to subtorus actions}\label{toricDown}
If $T\subseteq(\kk^*)^n$ occurs as a subtorus induced by a surjective
map $\deg:\Z^n\surj M$ (corresponding to the choice of 
degrees $\deg x_i\in M$),
then every affine toric variety $\toric{\delta}$ with $\delta\subseteq\Q^n$ 
inherits a $T$-action.  
By \cite[\S11]{toral}, the associated p-divisor $\pDiv(\delta)$
can be obtained as follows:
Defining $M_Y:=\ker(\deg)$, we have two mutually dual
exact sequences
$$
\xymatrix@R=4.5ex@C=2.2em{
0\ar[r] &
N \ar[rr]^-{i} &&
\Z^n \ar[rr]^\pi^-{} 
&&
N_Y \ar[r] 
\ar@/^1.1pc/[ll]_-s^-{} 
&
0\\
0 &
M \ar[l] &&
\Z^{n} \ar[ll]_-{\deg}^-{} \ar@/_1.1pc/[rr]^{s^*} &&
M_Y \ar[ll]_-{}^-{} 
&
0; \ar[l]
}
$$
with $s$ we denote a section of $\pi$.
Then, $\pDiv(\delta)$ lives on the toric variety
$Y:=\toric{\Sigma}\supseteq N_Y\otimes_\Z\kk^*=:T_Y$ 
with $\Sigma$ denoting the fan in $N_Y$ being the coarsest
common refinement of the image under $\pi$ of all faces of $\delta$.
As a function, $\pDiv(\delta)$ is given by
$$
\pDiv(\delta)(u) = s^*\big(\deg^{-1}(u)\cap\delta\dual\big)
$$
where the right hand side is a polyhedron in $M_Y$ whose normal fan is
refined by $\Sigma$. Thus, it encodes a semiample,
$T_Y$-invariant divisor on $Y$.
This implies that 
$$
\pDiv(\delta)=\sum_{a\in\Sigma(1)} \Delta_a\otimes\ko{\orb}(a)
\hspace{0.9em}\mbox{with}\hspace{0.6em}
\Delta_a=\big(\pi^{-1}(a)\cap\delta\big) - s(a)\subseteq N_\Q.
$$
Here $a\in\Sigma(1)$ are primitive lattice elements of rays in
$\Sigma$ and $\ko{\orb}(a)$ are their associated $T_Y$-invariant
divisors. The relation between these two representations of $\pDiv$
has been proved in \cite[Proposition 8.5]{toralv1} and, in a broader
context, in \cite{Diane}.

\subsection{The polyhedral coefficients}\label{inequDelta}
We will now present a method to describe the coefficients $\Delta_a$
with inequalities. This observation is as trivial as it is useful.

\begin{lemma}
\label{lem-inequDelta}
In the situation of {\rm (\ref{toricDown})}, the polyhedral coefficients
$\Delta_a$ are cut out by the inequalities
$\langle \kbb, \deg(r)\rangle \geq -\langle s(a), r\rangle$ for
$\,r\in\delta\dual$
(or generators of $\delta\dual$).
\vspace{-2ex}
\end{lemma}

\begin{proof}
$x\in \Delta_a \iff i(x)+s(a) \in \pi^{-1}(a)\cap\delta
\iff i(x)+s(a) \in \delta \iff 
\langle x, \deg(r)\rangle + \langle s(a), r\rangle =
\langle i(x) + s(a), r\rangle \geq 0$ for all $r\in\delta\dual$.
\end{proof}

\subsection{Toric Cox rings}\label{toriCox}
Let $\CF$ be a simplicial
fan in some lattice $N_Z$. Identifying again its one-dimensional
rays $\CF(1)=\{a^1,\ldots,a^n\}$ with the first lattice points sitting on them,
we assume that $\CF(1)$ generates $N_Z$.
We would like to apply Lemma \ref{lem-inequDelta} to understand
the Cox ring of the $\Q$-factorial
toric variety $Z:=\toric{\CF}$. As a ring, it is simply 
$\Cox(Z)=\kk[x_a\kst a\in\CF(1)]$; but by setting $M:=\Cl(Z)$, 
it is then the $M$-grading which makes it interesting.
The exact sequences from (\ref{toricDown}) become 
$$
\xymatrix@R=4.5ex@C=2.2em{
0\ar[r] &
\Cl(Z)^* \ar[rr]^-{i} &&
\Div_{\TT}^*Z \ar[rr]^\pi^-{} 
&&
N_Z \ar[r] 
\ar@/^1.1pc/[ll]_-s^-{} 
&
0\\
0 &
\Cl(Z) \ar[l] &&
\Div_{\TT}Z \ar[ll]_-{p}^-{} &&
M_Z \ar[ll]_-{\div}^-{} 
&
0. \ar[l]
}
$$
Here we have denoted by $\Div_{\TT}Z\cong\Z^n$ the group of $T_Z$-equivariant
divisors; the rays $a^i$ are the images of the unit vectors $e^i$. 
Note that the torus $T$ acting on $\Cox(Z)$ is the Picard torus
$T=\gHom(\Cl(Z),\kk^*)$.
The degree cone of $\Cox(Z)$ is the cone of effective
divisors
$\Eff(Z)\subseteq\Cl_\Q(Z)$. Hence, the tail of the p-divisor
$\pDiv_{\Cox}$ will be the 
dualized cone $\Eff(Z)\dual\subseteq\Cl_\Q(Z)^*$.
\\[1ex]
According to (\ref{toricDown}), $\pDiv_{\Cox}$ lives on
$Y:=\toric{\Sigma}$ with $\Sigma$ being the coarsest fan in $N_Y=N_Z$
containing all possible cones generated by subsets of $\CF(1)$. In particular,
$\Sigma$ is a subdivision of $\CF$, i.e.\ $\Sigma\finer\CF$, i.e.\
there is a proper map $\psi:Y\to Z$ that becomes 
an isomorphism if it is restricted on the tori $T_Y=T_Z$.
In the surface case we have $\Sigma=\CF$; hence $Y=Z$ and $\psi=\id$.
Finally, the choice of the section $s$ will not affect the upcoming result.

\subsection{The splitting of $\pDiv_{\Cox}$}\label{pToriCoxIdSurf}
While p-divisors on $Y$ may be altered by so-called principal p-divisors
coming from $N\otimes_\Z\kk(Y)^*=\gHom(M,\kk(Y)^*)$,
this does not mean that $\pDiv$ is determined by an element
of $\Pol(N_\Q,\sigma)\otimes_\Z\Pic_{{\Q}}(Y)$. However,
elements of the group $N\otimes_\Z\Pic^\Q(Y)=\gHom(M,\Pic^\Q(Y))$ 
with $\Pic^\Q(Y):=\CaDiv_\Q(Y)/\PDiv(Y)\neq\Pic(Y)\otimes_\Z\Q$ 
denoting the $\Q$-Cartier divisors modulo principal divisors
do indeed give a correct description of an equivalent class of a 
polyhedral divisor. In particular,
it makes sense to add those elements to already existing p-divisors.

\begin{definition}
\label{def-pToriCoxThSurf}
In the case of 
(\ref{toriCox}), the pull back map
$M=\Cl(Z)\subseteq\Pic^\Q(Z)\to\Pic^\Q(Y)$ defines an element
$\psi^*\in \gHom(M,\Pic^\Q(Y))= N\otimes_\Z\Pic^\Q(Y)$ giving
rise to a splitting
\fbox{$\pDiv_{\Cox}=\psi^*+\pDiv'_{\Cox}$} 
with some correction term $\pDiv'_{\Cox}$.
\end{definition}

\begin{remark}
Note that although $\Pic^\Q(Y)\ne \Pic_\Q(Y)$,
    we nevertheless have a map $\Pic^\Q(Y)=\CaDiv_\Q(Y)/\PDiv(Y)\to
    \CaDiv_\Q(Y)/\PDiv_\Q(Y)=\Pic_\Q(Y)$, and therefore $\pDiv_{\Cox}$
    determines a map $\Cl_\Q(Z)\supset\Eff(Z)\to
    \Nef(Y)\subset\Pic_\Q(Y)$.  
\\
  The splitting of $\pDiv_{\Cox}$ into $\psi^*$ and a correction term
  is then quite natural: Since, on the one hand, a p-divisor $\pDiv$
  encodes the ring $\bigoplus_{u\in M}\kG(Y,\pDiv(u))$ and, on the
  other, $\Cox(Z)=\bigoplus_{u\in M}\kG(Z,u)= \bigoplus_{u\in
    M}\kG(Y,\psi^*u)$, one is tempted to say that $\pDiv=\psi^*$.
  However, since
  $\tail{(\pDiv)}\dual=\Eff(Z)\supseteq\Nef(Z)$, it is generally
  the case that
  $\psi^*(\tail{(\pDiv)}\dual)\not\subseteq\Nef(Y)$, i.e.\
  $\psi^*$ is not a p-divisor.  Thus, all
  $u\in\tail{(\pDiv)}\dual\cap M$ leading to
  non-semiample divisors have to be processed.
\end{remark}

\subsection{The p-divisor of toric Cox rings}\label{pToriCoxTh}
If $E\subseteq{Y}=\toric{\Sigma}$ and $P\subseteq Z=\toric{\CF}$ are
toric prime divisors, then there are associated rays
$a(E)\in{\Sigma}{(1)}\subseteq N_{Y\kyz Z}:=N_Y=N_Z$ and 
$a(P)\in{\CF}{(1)}\subseteq
N_{Y\kyz Z}$, respectively.  Remember that we identify a ray with its
integral, primitive generator.  In particular, each $a(E)$ sits in a
unique minimal cone $C_E\in\CF$; hence there are unique
$\lambda_E(P)\in\Q_{>0}$ such that $a(E)=\sum_{a(P)\in C_E}
\lambda_E(P) a(P)$. (Remember that $\CF$ is a simplicial fan.) Set
$\lambda_E(P):=0$ for $a(P)\notin C_E$.

\begin{remark}
  Note that $\lambda_E(P)>0$ $\iff$ $a(P)\in C_E$ $\iff$
  $\psi(E)\subseteq P$.  In non-toric terms, these coefficients can be
  expressed as $\lambda_E(P)=\mult_E(\psi^*P)$.  If $E$ does not get
  contracted, then we may identify $E\subseteq{Y}$ with its divisorial
  image $\psi(E)\subseteq Z$; then $\lambda_E(P)=1$ if $E=P$ and
  $\lambda_E(P)=0$ of $E\ne P$ . In dimension two, this is always the
  case (because ${Y}=Z$).
\end{remark}

\begin{theorem}
\label{pToriCoxTheorem}
$\,\pDiv'_{\Cox}=\sum_E \Delta_E\otimes E$
with $E\subseteq{Y}$ running through the toric prime divisors
and $\Delta_E\subseteq\Cl(Z)^*_\Q$ being the polyhedron cut out by the 
inequalities
$\langle\kbb,[P]\rangle\geq -\lambda_{E}(P)$ for toric prime divisors $P$.
In particular $\Delta_E\supseteq\tail\pDiv_{\Cox}$.
\vspace{-2ex}
\end{theorem}

\begin{proof}
In the first exact sequence of (\ref{toriCox}),
we add the cosection $t:\Z^n\to\Cl(Z)^*$ induced from $s$.
Then, the maps satisfy $\,it+s\pi=\id_{\Z^n}$, $\,ti=\id_{\Cl^*}$, and 
$\,\pi s=\id_{N_{Y\kyz Z}}$.
$$
\xymatrix@R=4.5ex@C=2.2em{
0\ar[r] &
\Cl(Z)^* \ar[rr]^-{i} &&
(\Z^n=\Div^*_{\TT}Z) \ar[rr]^\pi^-{} 
\ar@/^1.1pc/[ll]_-{t} 
&&
N_{Y\kyz Z} \ar[r] 
\ar@/^1.1pc/[ll]_-s^-{} 
&
0
}
\vspace{1ex}
$$
Denote by $\{e(P)\}\subseteq \Z^n=\Div^*_{\TT}Z$ the dual basis
with respect to that of the toric prime divisors of $Z$.
In particular, $\pi(e(P))=a(P)\in N_{Y\kyz Z}$.
This notion can be extended to the prime divisors on ${Y}$ via
$e(E):=\sum_{a(P)\in C_E} \lambda_E(P) e(P)$;
we keep the property $\pi e=a$.
\\[1ex]
If $E\subseteq{Y}$ is a toric prime divisor (corresponding to the ray
$a(E)\in{\Sigma}{(1)}\subseteq N_{Y\kyz Z}$), then,
by Lemma \ref{lem-inequDelta}, the true coefficient
$\Delta_{E}^{\Cox}$ is given by the inequalities
$\langle \kbb, [P] \rangle \geq -\langle s(a(E)),\, P\rangle$
where the latter just means the $P$-th entry 
of $-s(a(E))\in\Z^n$.
On the other hand,
the claimed inequalities for $\Delta_E$ of $\pDiv'_{\Cox}$ 
are $\langle \kbb, [P] \rangle \geq -\lambda_E(P)
=-\langle e(E),P\rangle$.
Thus, it remains for us to show that $b(E):=e(E)-s(a(E))\in\Z^n$
is contained in $\Cl(Z)^*\subseteq\Z^n$ and satisfies
$d:=\sum_E b(E)\otimes [E]=\psi^*\in\Cl(Z)^*\otimes\Cl({Y})$.
\\[1ex]
The first claim follows from
$\,b(E)=e(E)-s(a(E))=e(E)-s\pi(e(E))=it(e(E))$.
Moreover,
$\,d=\sum_E it(e(E))\otimes [E]=
\big((it)\otimes \cl_{{Y}}\big)\circ
\big(\sum_E e(E)\otimes E\in\Div_{\TT}^*Z\otimes_\Z\Div_{\TT}{Y}\big)$
where $\cl$ denotes the canonical map $\,\Div\to\Cl$.
On the other hand,
since, for a toric prime divisor $P\subseteq Z$,
$\psi^\ast P=\sum_E \lambda_E(P) E
$,
we obtain that
$\psi^*=\sum_{E,P} \lambda_E(P)\, e(P)\otimes E=\sum_E e(E)\otimes E$,
i.e.\
$\,d=\big((it)\otimes \cl_{{Y}}\big)\circ \psi^*$. Restricted,
via $i$, to $\Cl(Z)^*$, this yields 
$\big((iti)\otimes \cl_{{Y}}\big)\circ \psi^*=\big(i\otimes
\cl_{{Y}}\big)\circ \psi^*=
\big(\cl_Z^*\otimes
\cl_{{Y}}\big)\circ \psi^*=\psi_{\Cl}^*$.
\vspace{-2ex}
\end{proof}
See (\ref{extwoP2}) for an example.

\bigskip

\section{Duality of polyhedra}\label{dualpol}%

\subsection{Cones over polyhedra}\label{dualPol}
Dualization of polyhedral cones 
via $\sigma\dual:=\{x\kst\langle\sigma,x\rangle\geq 0\}$
is a straightforward generalization of the
dualization of vector spaces. 
One has the basic relations $(\sigma\dual)\dual=\sigma$ and
$(\sigma_1\cap\sigma_2)\dual=\sigma_1\dual + \sigma_2\dual$.
Moreover, via $\tau\,(\leq \sigma) \mapsto
\tau':=\tau^\bot\cap\sigma\dual\,(\leq\sigma\dual)$
it provides a bijection of faces. 
For the convenience of the reader,
we will recall how this theory can be further extended to
the set of polyhedra containing the origin.
\\[1ex]
Let $V$ be a finitely-dimensional $\Q$-vector space and $\PD\subseteq
V$ be a polyhedron containing $0$. Then, we define
$$
\nabla:=\PD\dual:=\{x\in V^*\kst \langle \PD,x\rangle \geq -1\}.
$$
This construction can be understood by the ordinary duality notion of
cones.  It just requires a definition of the cone $C(\PD)$ spanned
over a polyhedron $\PD$ located in an affine hyperplane $V\times
\{1\}\subset V\times\bQ$. Namely, we set
$$
C(\PD):=\ko{\Q_{\geq 0}\cdot (\PD,1)}\;=\; 
\Q_{> 0}\cdot (\PD,1) {\;\textstyle \disjointcup\;} (\tail(\PD),0)
\;\subseteq\; V\oplus\Q.
$$
The polyhedron $\PD$ can be recovered as cross section $\PD=C(\PD)\cap
(V\times\{1\})$.  Then we verify that $C(\nabla %
)=C(\PD)\dual$; hence $\nabla\dual=(\PD\dual)\dual=\PD$ and
$\,(\PD_1\cap \PD_2)\dual =\conv(\PD_1\cup \PD_2)$.  Note that
$\PD_1+\PD_2\subseteq 2\conv(\PD_1\cup \PD_2)\subseteq 2(\PD_1+\PD_2)$
and, in general, $C(\PD_1+\PD_2)\neq C(\conv(\PD_1\cup
\PD_2))=C(\PD_1)+C(\PD_2)$.

\subsection{Heads and tails}\label{headTail}
Inside $V$ there are two cones associated to $\PD$. One is the already
mentioned $\,\tail(\PD)=C(\PD)\cap (\ku{0},1)^\bot$; 
since $0\in \PD$, we have $\tail(\PD)\subseteq \PD$.
The other is $\head(\PD):=\Q_{\geq 0}\PD\supseteq \PD$. 

If $\PD$ was already a polyhedral cone itself, then both cones
coincide and are equal to $\PD$.  In general, polyhedral duality
interchanges both constructions, i.e.\ $\tail (\nabla %
)=\head(\PD)\dual$ and $\head(\nabla)=\tail(\Delta)\dual$. Indeed,
$x\in\tail \PD\dual$ $\iff$ $\PD\dual +\Q_{\geq 0}\,x \subseteq
\PD\dual$ $\iff$ $\langle x,\PD \rangle \geq 0$ $\iff$ $\langle
x,\Q_{\geq 0}\PD \rangle \geq 0$. This duality is even more
transparent if we note that
$$\begin{array}{ccc}\head(\PD)=\bigcup_{t\rightarrow\infty}t\cdot\PD&
{\rm \ and\  }&
  \tail(\PD)=\bigcap_{t\rightarrow 0}t\cdot\PD\end{array}$$

\subsection{Face duality}\label{faceDual}
Via applying $C$, the nonempty faces $F\leq \PD$ correspond
bijectively to the faces of $C(\PD)$ not contained in $\tail(\PD)\leq C(\PD)$. 
The inverse map is the intersection with $V\times\{1\}$.
Since the dual face $(\tail \PD)'\leq C(\PD)\dual=C(\nabla %
)$ 
contains $(\ku{0},1)$,
it is not contained in $\tail \nabla %
$, and it corresponds to the minimal face of $\nabla %
$ that contains $0$.
Thus, restricting the duality
$\,\faces(C\PD)\leftrightarrow\faces(C\nabla %
)$ to those faces
with $\not\subseteq (\tail \PD)$ and $\not\supseteq (\tail \nabla %
)'$ 
on the left
hand side and doing similarly on the right, we obtain an order and dimension
reversing bijection
$$
\{\mbox{faces } F\leq \PD\kst 0\notin F\}
\leftrightarrow
\{\mbox{faces } F'\leq \nabla %
\kst 0\notin F'\}.
$$
The remainings of the bijection 
$\,\faces(C\PD)\leftrightarrow\faces(C\nabla %
)$
translate into
$$
\{\mbox{faces } F\leq \PD\kst 0\in F\}=
\faces(\head\PD) \leftrightarrow \faces((\head\PD)\dual)=
\faces(\tail \nabla %
)
$$
and, analogously, 
$\,\faces(\tail \PD)\leftrightarrow \{\nabla %
\mbox{-faces containing }0\}$.

\subsection{Fanwise linear functions}\label{pwlinear}
A rational (or real) function is called fanwise linear if it is linear
on the closed cones of a fan (hence it is continuous on the support of the
fan). This is equivalent to being piecewise (affine) linear and
homogeneous, that is $f(t\cdot v)=t\cdot f(v)$ for $t\in\bQ_{\geq 0}$.
For a polyhedron $\Delta\subseteq V$, we define the fanwise linear
function $\pwl(\Delta): V^*\rightarrow \bQ\cup\{-\infty\}$ by setting
$\pwl(\Delta)(v)=\min \langle \Delta,v\rangle$.  In particular,
$\pwl(\Delta)^{-1}(\bQ)=(\tail \Delta)\dual$.  If, additionally,
$0\in\Delta$, then $\pwl(\Delta): V^*\rightarrow \bQ_{\leq
  0}\cup\{-\infty\}$ with $\pwl(\Delta)^{-1}(\bQ_{\leq 0})
=\head(\nabla)$.  Moreover, $\pwl(\Delta)^{-1}(0)=\tail(\nabla)$.

\begin{lemma}\label{pwl-characterization}
  If $\Delta$ and $\nabla$ are mutually dual polyhedra containing $0$,
  then
$$\pwl(\Delta)(v)=\frac{-1}{\max\{t\in\bQ\kst tv\in\nabla\}}.$$
Equivalently, the homogeneous, continuous function
$\,\pwl(\Delta):\head(\nabla)\to\Q_{\leq 0}$ is characterized by the
property that $\,\pwl(\Delta)\equiv -1$ on
$\,\partial\nabla\cap\innt(\head\nabla)$, where $\partial$ and $\innt$
denote, respectively, relative boundary and interior of the cone.  In
particular, $\pwl(\Delta)$ is equal to $-1$ on all non-zero vertices
of $\nabla$.
\end{lemma}
\begin{proof}
  Let us consider $v\in \,\partial\nabla\cap\innt(\head\nabla)$, then
  $t\cdot v\not\in \nabla$ for every $t>1$.  Moreover, by definition,
  $\langle v,\Delta\rangle\geq -1$, hence $\pwl(\Delta)(v)\geq -1$. On
  the other hand, if $0> \lambda>-1$ is such that for all $u\in
  \Delta$ it holds $\langle u,v\rangle \geq \lambda$, then $\langle
  u,|\lambda|^{-1}v\rangle\geq -1$; hence, by definition of duality of
  polyhedra, $|\lambda|^{-1}v$ is in $\nabla$ contradicting the
  assumption.
\end{proof}

Conversely, let $f:\altcone\rightarrow \bQ_{\geq 0}$ be a fanwise
linear function defined on a rational, convex polyhedral cone
$\altcone\subseteq V^*$.  We assume that $f$ is also concave, that is
$f(v_1+v_2)\leq f(v_1)+f(v_2)$. Defining
$$
\nabla_f:=\conv\big\{ f(v)^{-1}\cdot v\kst v\in\altcone\big\}
\hspace{0.8em}\mbox{with}\hspace{0.6em}
0^{-1}\cdot v:=\bQ_{\geq 0}\cdot v,
$$
we get a polyhedron with $\,\head(\nabla_f)=\altcone$ and
$\,\tail(\nabla_f)=f^{-1}(0)$.

\begin{lemma}\label{pwl<->polyhedra}
  Let $\Delta_f$ be a polyhedron dual to $\nabla_f$ defined
  above. Then, over the cone $\altcone\subseteq V^*$ it holds
     $$\pwl(\Delta_f)=-f.$$
\end{lemma}
\begin{proof}
  Let us set $g(v)=\big(\sup\{t\kst tv\in\nabla_f\}\big)^{-1}$.  Clearly,
  both $f$ and $g$ vanish exactly on $\tail(\nabla_f)\subset \sigma$
  so we can assume that $v$ is chosen so that both are non-zero. By
  definition of $\nabla_f$ we have $f(v)^{-1}\cdot v\in\nabla_f$; hence
  $f(v)^{-1}\leq \sup\{t\kst tv\in\nabla_f\}$ and thus $g(v)\leq
  f(v)$. Now suppose that $t\cdot v\in\nabla_f$; hence, by definition
  of $\nabla_f$, $$t\cdot v=\sum_i a_if(v_i)^{-1}\cdot v_i$$ for some
  $v_i\in\sigma$ and positive numbers $a_i$ such that $\sum_i
  a_i=1$. Applying the function $f$ to both sides of the equality and
  using its homogenity and convexity, we get $$t\cdot f(v)\leq \sum_i
  a_if(v_i)^{-1}\cdot f(v_i),$$ hence $t\cdot f(v)\leq 1$. Thus
  $\sup\{t\kst tv\in\nabla_f\}\leq f(v)^{-1}$ hence $g(v)^{-1}\leq
  f(v)^{-1}$ and thus $g(v)\geq f(v)$. Since $g=-\pwl(\Delta_f)$ this
  concludes the proof.
\end{proof}

\begin{remark}
  It is possible to weaken the assumption of fanwise linearity to
  homogeneity of $f$ (see the above proof). Then, $\nabla_f$ and
  $\Delta_f$ still become well-defined, mutually dual convex bodies --
  but they lose their polyhedral structure.
\end{remark}

\subsection{Dualized Cox coefficients}\label{dualPolCox}
The duality described in (\ref{dualPol}) allows a nicer description of the
polyhedral coefficients $\Delta_E\subseteq\Cl_\Q(Y)^*$ from
Theorem \ref{pToriCoxTheorem}. Since they contain the origin,
it makes sense to define their duals
$\nabla_E:=\Delta_E\dual\subseteq\Cl_\Q(Z)$. It follows that
$$
\renewcommand{\arraystretch}{1.3}
\begin{array}{@{}r@{\;}c@{\;}l@{}}
\nabla_E  &=&
\conv\big\{0,\, [P]/\lambda_E(P)\kSt 
\psi(E)\subseteq P\subseteq Z\big\}
+ \hspace{-0.1em}\sum_{P\not\supseteq\psi(E)}\hspace{-0.2em}
\Q_{\geq 0}\cdot [P]
\\
&=&
\conv\big\{ [P]/\lambda_E(P)\kSt 
P\subseteq Z\big\} \;\subseteq\;\head\nabla_E=\Eff(Z)
\hspace{0.7em}
(v/0:=\bQ_{\geq 0}\cdot v)
\end{array}
$$
with $P$ running through the toric prime divisors of $Z$
and
$\lambda_E(P)=\mult_E(\psi^*P)$.
Using these polyhedra, we obtain
$\pDiv_{\Cox}'=\sum_E\nabla_E\dual\otimes E$,
and $\pDiv_{\Cox}'(u)$ contains $E$ with multiplicity
$$
\min\langle\Delta_E,u\rangle=
-1/\max\{\lambda\in\Q\kst \lambda u\in\nabla_E\} \;\in\; 
\Q_{\leq 0}\cup\{-\infty\}.
$$

\section{MDS and their Cox p-divisor} \label{MDS}

\subsection{Mori dream spaces}\label{introMDS}
Mori dream spaces (MDS) were introduced in \cite{MDS}. Recall that $Z$
is a $\Q$-factorial variety with $\Cl(Z)$ being a lattice and
$\Cox(Z)$ being finitely generated.

The birational geometry of $Z$ is finite, i.e.\ $Z$ has finitely many
small (i.e.~isomorphic in codimension one) $\Q$-factorial
modifications $Z_i$ (set $Z_0:=Z$); we will call them SQM models of
$Z$.  The varieties $Z_i$ are exactly the $\Q$-factorial GIT quotients
of $\Cox(Z)$ by the Picard torus arising from
  linearizations of the trivial bundle depending on the choice of a
  character of the torus, see \cite{MDS}.
All models $Z_i$ share the same Cox ring and can be distinguished by
pure combinatorics, cf.\ \cite{CoxII}.  In particular, by strict
transforms, we can identify $\Div(Z_i)$ and $\Cl(Z_i)$ with $\Div (Z)$ and
$\Cl (Z)$, respectively.  The same holds true for the cones $\Eff
(Z_i)=\Eff (Z)$ and $\Mov (Z_i)=\Mov (Z)$.  However, the cones $\Nef(Z_i)$ are
different, that is $\innt\Nef(Z_i)\cap\innt\Nef(Z_j)=\emptyset$ if
$Z_i\ne Z_j$, and we have the decomposition $\Mov(Z)=\bigcup_i\Nef(Z_i)$,
\cite[1.11(3)]{MDS}.  This chamber decomposition is polyhedral and
coincides with that of the stability with respect to the Picard torus,
cf.\ \cite[2.3]{MDS} and \cite{vgit}. Finally, the maybe most striking
feature of Mori dream spaces is that nefness implies semiampleness.

\subsection{The Chow limit}\label{introChow}
Let $Y$ be the Chow quotient of $\Cox(Z)$ by the Picard torus, i.e.,
by abuse of notation, the normalized component of the inverse limit of
the models (GIT quotients) $Z_i$ that is birational to the original
$Z$. In particular, we have birational morphisms $\psi_i:Y\to Z_i$.

Note that $Y$ carries two types of exceptional divisors:
\vspace{-2ex}
\begin{itemize}
\item[(i)] An irreducible divisor $E\subseteq Y$ is called of the
  first kind if it is a component of the exceptional locus of a
  morphism $\psi_i:Y\rightarrow Z_i$. Note that since $Z_i$ is
  $\bQ$-factorial, the exceptional locus of $\psi_i$ is of pure
  codimension 1. 
  Moreover, since the $Z_i$ are isomorphic outside
  codimension 2, the set of exceptional divisors is the same for all
  $\psi_i$.
\item[(ii)] We say that an irreducible divisor $E$ is an exceptional
  divisor of the second kind if it is a strict transform to $Y$ of a
  (divisorial) component of an exceptional locus of a birational
  morphism (divisorial contraction) of a $Z_i$. In other words,
  cf.~\cite[1.11(5)]{MDS}, $E$ is a strict transform of a non-movable
  divisor from $Z$.
\end{itemize}

\subsection{Stabilized multiplicities}\label{multStab}
Let $\psi:Y\to Z$ be a proper, birational morphism and $E\subseteq Y$
a prime divisor. Then, in the toric case we used in
(\ref{pToriCoxTh}) and (\ref{dualPolCox}) the muliplicities
$\lambda_E(P)=\mult_E(\psi^*P)$ of a divisor $\psi^*P$ in the general point
of $E$ in $Z$.%
\\[1ex]
In \cite[\S2]{asymptotic} there is a stable version
of these multiplicities. At least for big divisors $P$,
one defines $\mult_E^{\stable}(\psi^*P)$ 
either as the $E$-multiplicity of the stable base locus of $P$ or,
by \cite[Lemma 3.3]{asymptotic}, as
$$
\mult_E^{\stable}(\psi^*[P]):= \inf_{D\in |P|_\Q} \mult_E(\psi^*D)
\leq \mult_E(\psi^*P).
$$
Here $D\in|P|_\Q$ means that $D$ is an (effective) $\Q$-divisor
with $mD\in|mP|$ for $m\gg 0$.
Finally, it follows from \cite[Theorem D]{asymptotic} that
for a Mori Dream Space $Z$ the stable multiplicity function
$\mult_E^{\stable}:=\mult_E^{\stable}\circ\psi^*$ 
can be extended to a concave, fanwise linear function
on $\Eff(Z)\subseteq\Cl(Z)_\Q$. 
We have the following immediate consequence of Lemma
\ref{pwl<->polyhedra}.
\begin{corollary}
\label{cor-multDelta}
Let $Z$ be an MDS and $\psi: Y\rightarrow Z$ the birational morphism
from the Chow quotient of $\Cox(Z)$. Let $E\subseteq Y$ a prime
divisor.  Then \vspace{-0.2ex}
$$
\textstyle
\renewcommand{\arraystretch}{1.7}
\begin{array}{rcl}
\nabla_E
&:=&
\conv\big\{\frac{[P]}{\mult^{\stable}_E\psi^*[P]}
\kSt [P]\in\Eff Z\big\} \subseteq\Cl(Z)_\Q
\vspace{-1.5ex}
\end{array}
$$
and
$$
\textstyle
\renewcommand{\arraystretch}{1.7}
\begin{array}{rcl}
\Delta_E
&:=&
\{C\in\Cl^*(Z)_\Q\kst C\geq -\mult^{\stable}_E\}
\vspace{-1.5ex}
\end{array}
$$
are mutually dual polyhedra with $\pwl(\Delta_E)=-\mult^{\stable}_E$.
Moreover, if $Z$ is toric, then they coincide with those from 
{\rm (\ref{dualPolCox})}.
\vspace{-2ex}
\end{corollary}

\subsection{The Cox p-divisor of an MDS}\label{mainRes}
Now we are able to present the p-divisor $\pDiv_{\Cox}$ describing
the Cox ring of an MDS. As in Definition \ref{def-pToriCoxThSurf},
we split $\pDiv_{\Cox} = \psi^* + \pDiv_{\Cox}'$.

\begin{theorem}
\label{th-mainRes}
The part $\pDiv_{\Cox}'$ of the p-divisor 
of the Cox ring of a MDS equals
$$
\pDiv_{\Cox}'= 
\sum_{E\subset Y}\Delta_{E}\otimes E,
$$ 
where the coefficients $\Delta_{E}$ are defined in Corollary
\ref{cor-multDelta}, and the sum is formally taken over all divisors
$E\subset Y$.  However, if $E$ is not one of the finitely many
exceptional divisors from {\rm (\ref{introChow})(i)} or {\rm (ii)},
then the corresponding coefficient is trivial, i.e.\
$\Delta_E=\tail\pDiv=\Eff(Z)\dual\subseteq\Cl(Z)^*_\Q$, anyway.
\vspace{-2ex}
\end{theorem}

\begin{proof}
  We will treat all SQM models $Z_i$ on equal footing, i.e.\ we
  consider $\pDiv_i:=\psi_i^*+\pDiv'_i$ with $\pDiv'_i:=\sum_{E\subset
    Y} \Delta_E^i\otimes E$ and $\Delta_E^i:= \{C\in\Cl^*(Z_i)_\Q\kst
  \langle C,\,[P]\rangle \geq -\mult_E\psi_i^*P\}$.  Since the
  divisors on $Z_i$ are identified, via the strict transform, with
  those on $Z$, we can compare the $\pDiv_i$ as functions
  $\pDiv_i:\Div(Z)=\Div(Z_i)\to \CaDiv_\Q(Y)$.  Taking, as we did in
  Corollary~\ref{cor-multDelta}, the function
  $\mult_E^{\stable}\circ\psi_i^*$ for the fanwise linear map $f$ in
  (\ref{pwlinear}), we obtain that $\pDiv_i(D)=\psi_i^*(D)
  -\sum_{E\subset Y} \mult_E^{\stable} \psi_i^*(D)\cdot E$ for
  $D\in\Div Z$.

  We claim that $\pDiv_i(D)=\pDiv_j(D)$. Indeed, since the
  multiplicities of $D$ along divisors $E$ contained in $Z$ (isomorphic
  in codimension 1 to $Z_i$ and $Z_j$) are the same, we conclude that
  the difference $\pDiv_i(D)-\pDiv_j(D)$ is supported on divisors
  contracted by $\psi$; more precisely we get
$$
\begin{array}{rcl}
\pDiv_i(D)-\pDiv_j(D)&=&\left(\psi_i^*(D)-\sum_{E\subset Exc(\psi)} 
\mult_E^{\stable} \psi_i^*(D)\cdot E\right)\\
&-&\left(\psi_j^*(D)-\sum_{E\subset Exc(\psi)} 
\mult_E^{\stable} \psi_j^*(D)\cdot E\right).
\end{array}
$$
But $\psi_i^*(D)-\sum_{E\subset Exc(\psi_i)} \mult_E \psi_i^*(D)\cdot
E$ is the strict transform of the $\Q$-Cartier divisor $D$ from $Z_i$ to
$Y$ via birational $\psi_i: Y\rightarrow Z_i$; hence, again by
isomorphism in codimension 1, it is the same for $\psi_j:Y\rightarrow
Z_j$. Thus, passing to the limit from $\mult_E$ to $\mult_E^{\stable}$, we
get the conclusion of our claim.

Let us recall that, by \cite[Prop 1.11(5)]{MDS}, every big divisor
$D\in \Div Z$, possibly replaced by its multiple, admits a canonical
splitting $\,D=\mov(D)+\fix(D)$ into the stable movable and fixed
part, respectively.  Moreover, there is an SQM model $Z_i$ such that
$\,\mov(D)\in \Nef(Z_i)$, i.e.\ $\mov(D)$ is semiample on $Z_i$. Thus,
the linear system $|\mov(D)|$ can be assumed base-point-free so that
it defines a contraction of $Z_i$ such that the support of $\fix(D)$
is in the exceptional locus of the contraction.  If $E_\nu\subseteq
Z_i$ denote divisors contracted by $|\mov(D)|$, then, by definition,
$\fix(D)=\sum_\nu \mult_{E_\nu}^{\stable}(D)\cdot E_\nu$.  We note
that we can write $\mult_{E_\nu}^{\stable}(D)= \mult_{E_\nu}(D)$
because $|\mov(D)|$ is base-point-free and $D\in |D|=|\mov(D)|$ can be
chosen general. Thus,
$$
\renewcommand{\arraystretch}{1.3}
\begin{array}{rcl}
\psi_i^*(D)
&=&
\psi^*_i(\mov (D)) + \sum_\nu \mult_{E_\nu}^{\stable}(D)\cdot
\psi^*_i(E_\nu)\\
&=&
\psi^*_i(\mov (D)) + \sum_\nu \mult_{E_\nu}^{\stable}(D)\cdot
\big(\widehat{E_\nu} + \sum_{E\subset Exc(\psi_i} 
\mult_E(\psi_i^*E_\nu)\cdot E\big)
\end{array}
$$
with $\widehat{E_\nu}\subseteq Y$ denoting the strict transform via
$\psi_i^*$ of $E_\nu$, i.e.\ being an exceptional divisor of the
second kind, and the second summation is restricted to exceptional
divisors of the first kind only.  In particular,
$\psi_i^*(D)-\psi^*_i(\mov (D))$ is supported exclusively on
exceptional divisors (of both kinds). On the other hand, as the pull
back of a semiample divisor, $\psi^*_i(\mov (D))$ does not contain
exceptional components at all when $D$ is general in its linear
system.  Thus,
$$
\renewcommand{\arraystretch}{1.3}
\begin{array}{rcl}
\psi_i^*(D)
&=&
\psi^*_i(\mov D) + \sum_{E\subset Y} \mult_E^{\stable}(\psi_i^*D)
\cdot E,
\end{array}
$$
and therefore, if $\pDiv(D)$ denotes the mutually equal
  $\pDiv_i(D)$, we obtain that $\pDiv(D)=\psi^*_i(\mov (D))$, and
$\pDiv(D)$ inherits the semiampleness from $\mov(D)$ on $Z_i$.

Eventually, since $|\mov(D)|=|D|$ the natural inclusion map
$\iota_i:\kG(Y,\pDiv(D))=\kG(Y,\psi^*_i(\mov D))\to
\kG(Y,\psi^*_i(D))=\kG(Z,D)$ becomes an isomorphism.  Since both maps
$D\mapsto\pDiv(D)$ and $D\mapsto\psi_i^\ast(D)-
\sum_E\mult^{\stable}_E(\psi_i^*D)\cdot E$ are piecewise linear, this
extends to the whole effective cone being the closure of the cone of
big divisors, cf.\ \cite[Theorem 2.2.26]{LazI}.  In particular,
$\pDiv$ is a decent p-divisor with $\kG(Y,\pDiv(D)) \to\kG(Z,D)$ being
an isomorphism for every $D\in\Eff(Z)\cap\Cl(Z)$; hence
$$
\bigoplus_{D\in Cl(Z)}\kG(Z,D) = \bigoplus_{D\in Cl(Z)}\kG(Y,\pDiv(D))
$$
gives a presentation of $\Cox(Z)$ as a p-divisor.
\end{proof}

The arguments in the proof of Theorem \ref{th-mainRes} yield the
following observation (cf.~the remark following 
Definition~\ref{def-pToriCoxThSurf}).
\begin{corollary}\label{Cox-pwl-map}
  The fanwise linear map $\pDiv_{\Cox}:\Eff(Z)\to\nef(Y)$ associated
  to p-divisor $\pDiv_{\Cox}$ is a composition of a fanwise linear
  retraction $\Eff(Z)\to\Mov(Z)$ and a fanwise linear map
  $\Mov(Z)\to\Nef(Y)$ whose restriction to the cone $\Nef(Z_i)$, for
  every SQM model $Z_i$, coincides with the pull-back map $\psi^*_i:
  \Nef(Z_i)\to\Nef(Y)$.
\end{corollary}

\subsection{Example: Blowing up two points in $\PP^3$}\label{extwoP3}
This is perhaps the simplest three-dimensional
example to illustrate Corollary
\ref{Cox-pwl-map}.
Let $Z$ be the blow-up of $\bP^3$ in two points,
say $x_1$ and $x_2$, with exceptional divisors denoted by $E_1$ and
$E_2$. The strict tranform of a general plane, a plane passing through
each of these points, and a plane passing through both of them, define
divisors whose classes span $\Mov(Z)$. The rational maps defined by
these divisors are onto $\bP^3$, $\bP^2$ and $\bP^1$, respectively. The
flop along the strict tranform of the line passing through $x_1$ and
$x_2$ yields another SQM model, let us call it $Z_1$. 
The variety $Y$ results from blowing up this strict tranform.

Now the following picture presents sections of cones in spaces of
divisor classes. The 3-dimensional cone $\Eff(Z)$ presented on the
left hand side gets retracted to $\Mov(Z)$: the regions on which the
retraction is linear are denoted by dotted line segments. Next
$\Mov(Z)=\Nef(Z)\cup\Nef(Z_1)$ is mapped linearly on each $\Nef$ cone to
two 3-dimensional faces of the 4-dimensional cone $\Nef(Y)$.
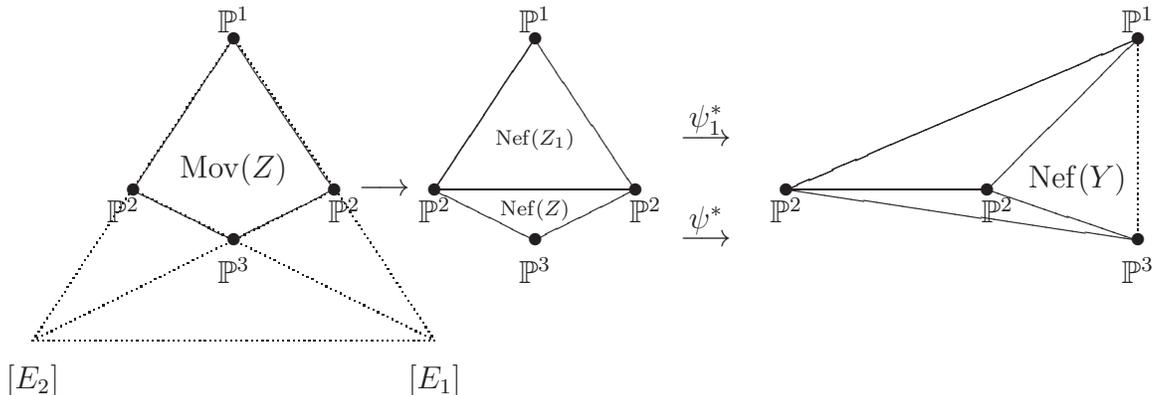
$$
\begin{xy}<38pt,0pt>:
(0,-0.4)*={[E_2]}, (2,3.2)*={\bP^1},
(4,-0.4)*={[E_1]}, (3.1,1.3)*={\bP^2},
(0.9,1.3)*={\bP^2}, (2,0.7)*={\bP^3}, 
(0,0)*={}="e2" ; (4,0)*={}="e1" **@{.},
(2,3)*={\bullet}="e0" ; "e1" **@{.},
"e0" ; "e2" **@{.},
(1,1.5)*={\bullet}="a" ; "e1" **@{.}, 
(3,1.5)*={\bullet}="b"; "e2" **@{.}, %
(2,1)*={\bullet}="h" ; "b" **@{-},
"h" ; "a" **@{-}, "e0" ; "a" **@{-}, "e0" ; "b" **@{-},
(2,1.7)*={\Mov(Z)},
(3.5,1.5)*={\longrightarrow},
(5,3.2)*={\bP^1},
(6.1,1.3)*={\bP^2},
(4,1.3)*={\bP^2}, (5,0.7)*={\bP^3}, 
(5,3)*={\bullet}="e01" ; (4,1.5)*={\bullet}="a1" **@{-}, 
(6,1.5)*={\bullet}="b1"; "a1" **@{-},
(5,1)*={\bullet}="h1" ; "b1" **@{-},
"h1" ; "a1" **@{-}, "e01" ; "a1" **@{-}, "e01" ; "b1" **@{-},
(5,2)*={\scriptstyle\Nef(Z_1)}, (5,1.3)*={\scriptstyle\Nef(Z)},
(6.7,1)*={\longrightarrow}, (6.7,1.2)*={\psi^*}, 
(6.7,2)*={\longrightarrow}, (6.7,2.2)*={\psi_1^*},
(11,3.2)*={\bP^1},
(9.6,1.3)*={\bP^2},
(7.5,1.3)*={\bP^2}, (11,0.7)*={\bP^3}, 
(11,3)*={\bullet}="e02" ; (7.5,1.5)*={\bullet}="a2" **@{-}, 
(9.5,1.5)*={\bullet}="b2"; "a2" **@{-},
(11,1)*={\bullet}="h2" ; "b2" **@{-},
"h2" ; "a2" **@{-}, "e02" ; "a2" **@{-}, "e02" ; "b2" **@{-},
(10.4,1.6)*={\Nef(Y)}, 
"e02" ; "h2" **@{.}
\end{xy}
$$
We note that only two of the four faces of the tetrahedron
representing the section of the 4-dimensional cone $\Nef(Y)$ are
associated to SQM models of $Z$. The other two faces represent
contractions of $Y$ to $\bP^3$ blown up at one point ($x_1$ or $x_2$)
and then along the strict transform of the line passing through $x_1$
and $x_2$. This is equivalent to blowing up the line first and then
blowing up the fiber of the exceptional divisor above $x_1$ or $x_2$. In
particular, the dotted edge of the tetrahedron represents 
the contraction of
$Y$ to $\bP^3$ blown up along the line passing through $x_1$ and
$x_2$.

\section{Surfaces} \label{surf}

\subsection{Specializing the general result}\label{genSurf}
The case of ($\Q$-factorial MD) surfaces $Z=S$ is special for two
reasons. First, it does not require the pull back to the Chow
quotient, i.e.\ $Y=Z=S$ with $\,\psi=\psi_i=\id$, and $\cD:
\eff(S)\rightarrow\nef(S)$ simply reflects the Zariski decomposition.
Indeed, given any effective divisor $D$ on $S$, we can write it
uniquely as the sum $D\equiv P+\sum_i a_iE_i$ where $P\in\Nef(S)$,
$E_i$ are exceptional curves (if there are any) such that $(P\cdot
E_i)=0$, and coefficients $a_i=\mult^{\stable}_{E_i}D$. Thus $P=\pDiv(D)$.

\vspace{-2ex} Second, the $\Q$-valued intersection product, denoted
simply by a dot, allows one to identify vector spaces
$\Cl(S)^*_\Q=\Cl(S)_\Q$ with $\langle C_1,C_2\rangle =(C_1\cdot
C_2)$. In particular, the polyhedral coefficients $\Delta_E$ will be
contained in $\Cl(S)_\Q$ now and have $\Nef(S)=\Eff(S)\dual$ as their
common tail cone.  
If $S$ is smooth, then we even know that
  $\Cl(S)^*=\Cl(S)$.  In general, this equation has to be replaced by
  $\Cl(S)^*=\{D\in\Cl(S)_\Q\kst \langle D,\Cl(S)\rangle\subset\Z\}$.  
Finally, we recognize the (finitely
many) exceptional divisors $E_i\subseteq S$ by their negative self
intersection numbers $(E_i^2)$.

\begin{theorem}
\label{pToriCoxTheoremSurf}
Let $S$ be an MD-surface with the exceptional divisors
$E_i\subset S$. Then,
$\pDiv'_{\Cox}=\sum_i \Delta_i\otimes E_i$
with 
$$
\Delta_i=\{D\in\Eff(S)\kst 
(D\cdot E_i)\geq -1\ \ 
{\rm and}\ \ (D\cdot E_j)\geq 0\ \ {\rm for}\ \ 
j\ne i \},
$$
and the dual coefficients equal
$
\,\nabla_i = \ko{0E_i} + 
\sum_{j\neq i} \Q_{\geq 0} [E_j] +\Nef(S).
$
\\
{\rm (With $\ko{0E_i}$ we denote the line segment
connecting $0$ and $[E_i]$ inside $\Cl_\Q(S)$.)}
\vspace{-2ex}
\end{theorem}

\begin{proof}
  This is a reformulation of Theorem \ref{th-mainRes}. $\Delta_i$ and
  $\nabla_i$ are dual with respect to the intersection product.  On
  the other hand, by Lemma \ref{pwl-characterization} the function
  defined by $\nabla_i$ is just $-\mult^{\stable}_{E_i}$.
\end{proof}

\subsection{Del Pezzo surfaces}\label{delPezzo}
Let $S=S_d$ be a smooth del Pezzo surface of degree $d=K_S^2$. By
definition, $-K_S$ is ample. Any such $S_d$ is known to be
$\bP^1\times\bP^1$ ($d=8$) or a blow-up of $\bP^2$ at $\,r:=9-d$
general points. It is known that for $d\leq 7$ the cone $\Eff(S)$ is
generated by a finite number of $(-1)$-curves. In fact, any
non-movable curve on such $S$ is a $(-1)$-curve. 
In this special case, the polyhedral coefficients 
from Theorem \ref{pToriCoxTheoremSurf}
become especially easy:

\begin{corollary}
\label{pToriCoxCor}
If $S$ is a del Pezzo surface
and $E$ a
$(-1)$-curve, then the only vertices of $\Delta_{E}$ are $0$ and
$[E]$.  In particular, the polyhedral coefficients of $\pDiv_{\Cox}'$
are as follows: $\Delta_E=\conv\{0,[E]\}+ \Nef(S)= \ko{0[E]}+\Nef(S)$.
\vspace{-2ex}
\end{corollary}
\begin{proof}
  If $D\in\Delta_E$, i.e.\ %
  if $D$ is an effective $\Q$-divisor with $(D\cdot E)\geq -1$ and
  $(D\cdot F)\geq -\delta_{EF}$ for $(-1)$-curves $F\ne E$, then we
  have to show that $D\in\ko{0E}+\Nef(S)$.  If $D$ was already nef,
  then we are done. If not, then by rescaling we may assume that
  $(D\cdot E)=-1$, and then we claim that $D':=D-E$ is nef: First,
  $(D'\cdot E)=(D\cdot E)-(E^2)=0$. Then, if $F$ is an arbitrary
  $(-1)$-curve different from $E$, we may write $D=eE+fF+P$ with
  $e,f\geq 0$ and $P$ being effective without $E$ and $F$
  contributions.  Thus,
$$
-1 = (D\cdot E) = -e +f(F\cdot E) + (P\cdot E) \geq -e +f(F\cdot E);
$$
hence $e-1\geq f(F\cdot E)\geq 0$.  This implies that $D'$ is
effective and, moreover,
$$
(D'\cdot F)
\geq (e-1)(E\cdot F)-f \geq f(E\cdot F)^2-f = f\big((E\cdot F)^2-1\big).
$$
If $(E\cdot F)\neq 0$, then we obtain $(D'\cdot F)\geq 0$; in the opposite
case of $(E\cdot F)=0$, 
we simply conclude via $(D'\cdot F)=(D\cdot F)-(E\cdot F)=(D\cdot F)\geq 0$.
\end{proof}

\begin{remark}
  Let $S=S_d$ be a smooth del Pezzo surface of degree $d\leq 7$ which
  is a a blow-up of $\bP^2$ at $\,r:=9-d$ general points; by
  $E_1,\ldots,E_r\subset S$ we denote their preimages.  Then
  $\Cl(S)=\Z H \oplus (\oplus_{i=1}^r \Z E_i)$, hence $\id_{\Cl
    S}=[H]\otimes [H]-\sum_{i=1}^r [E_i]\otimes [E_i]$.  In
  particular, \vspace{-1ex}
$$
\pDiv_{\Cox} = \big([H]+\Nef(S)\big)\otimes H + 
\sum_{i=1}^r \big(\ko{[-E_i]0}+\Nef(S)\big)\otimes E_i
+ \hspace{-0.3em}\sum_{E\notin\{\!E_i\!\}}\hspace{-0.5em} 
\big(\ko{0[E]}+\Nef(S)\big)\otimes E.
\vspace{-2ex}
$$

The above result says that Zariski decomposition on a del Pezzo
surface is orthogonal. That is, given any effective divisor $D$ on $S$,
we can write it uniquely as the sum $D\equiv P+\sum_i a_iE_i$ where
$P\in\Nef(S)$ and $E_i$ are $(-1)$-curves such that $(P\cdot E_i)=0$,
and $a_i=\mult_{E_i}D$ and, moreover, $(E_i\cdot E_j)=0$ if $i\ne j$.
The last of these properties is known and follows from the fact that
the birational morphism of a del Pezzo surface associated to $|mP|$,
$m\gg 0$ contracts disjoint $(-1)$-curves $E_i$. 
\end{remark}

\subsection{Example: Blowing up two points in $\PP^2$}\label{extwoP2}
While the following two examples are just toric, they, nevertheless,
illustrate the special shape of $\pDiv_{\Cox}$ for del Pezzo surfaces
and indicate the difference to a somewhat more general
situation. First, we consider a surface $S_1$ which is an ordinary
blowing up of $\PP^2$ in two points; second we present a surface $S_2$
which is a $\PP^2$ with two infinitesimally near points blown up.
\\[1ex]
The toric surface $S_1$ is given by the fan
$\,\Sigma_1=\{(1,0),\; (1,1),\; (0,1),\; (-1,0),\; (-1,-1)\}$. 
The exceptional divisors of the blowing up are $E_1=\ko{\orb}(1,1)$
and $E_2=\ko{\orb}(-1,0)$ together with
the strict transform $E_0=\ko{\orb}(0,1)$ of the line connecting the
two centers; they are the only $(-1)$-curves in $S_1$.
\\[1ex]
Let $[H]$ denote the pull back of the line in $\PP^2$.
Then $[E_0]=[H]-[E_1]-[E_2]$,
and  the nef cone $\Nef(S_1)$ is formed by the strict transforms
$[A]=[H]-[E_1]=[E_0]+[E_2]$, 
$\,[B]=[H]-[E_2]=[E_0]+[E_1]$, and by $[H]=[E_0]+[E_1]+[E_2]$ itself. The 
ample anti-canonical bundle is $[-K]=3[H]-[E_1]-[E_2]=[A]+[B]+[H]$.
\\[1ex]
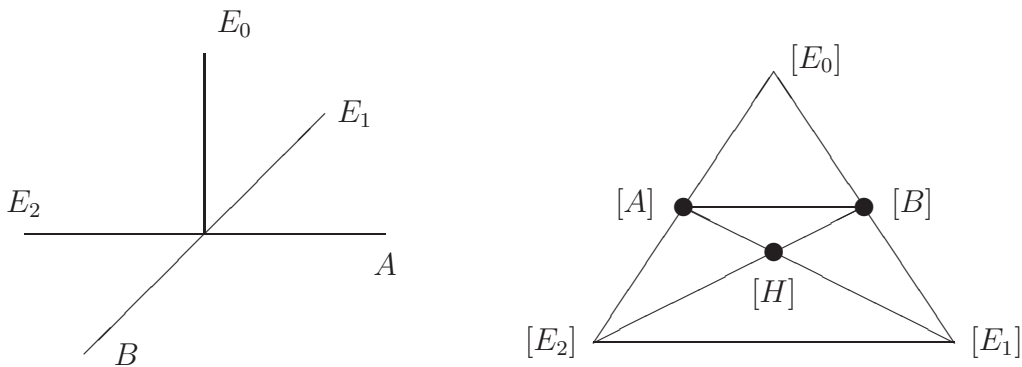
\begin{minipage}{15cm} 
\unitlength=0.8mm
\definecolor{cblowup}{rgb}{1.0,.388,.278}
\linethickness{0.4pt}
\hspace*{\fill}
\begin{picture}(70.00,70.00)(0,0)
\put(30.00,30.00){\line(1,0){30.00}}
\put(30.00,30.00){\line(1,1){20.00}}
\put(30.00,30.00){\line(0,1){30.00}}
\put(30.00,30.00){\line(-1,0){30.00}}
\put(30.00,30.00){\line(-1,-1){20.00}}
\put(55.00,50.00){\makebox(0,0)[cc]{$E_1$}}
\put(35.00,65.00){\makebox(0,0)[cc]{$E_0$}}
\put(0.00,35.00){\makebox(0,0)[cc]{$E_2$}}
\put(60.00,25.00){\makebox(0,0)[cc]{$A$}}
\put(17.00,10.00){\makebox(0,0)[cc]{$B$}}
\end{picture}
\hspace*{\fill}
\hspace*{\fill}
\begin{picture}(70.00,60.00)(0,-12)
\put(0.00,0.00){\line(1,0){60.00}}
\put(0.00,0.00){\line(2,3){30.00}}
\put(60.00,0.00){\line(-2,3){30.00}}
\put(0.00,0.00){\line(2,1){45.00}}
\put(60.00,0.00){\line(-2,1){45.00}}
\put(45.00,22.50){\circle*{3.00}}
\put(15.00,22.50){\circle*{3.00}}
\put(30.00,15.00){\circle*{3.00}}
\put(15.00,22.50){\line(1,0){30.00}}
\put(-7.00,00.00){\makebox(0,0)[cc]{$[E_2]$}}
\put(67.00,00.00){\makebox(0,0)[cc]{$[E_1]$}}
\put(37.00,47.00){\makebox(0,0)[cc]{$[E_0]$}}
\put(53.00,23.00){\makebox(0,0)[cc]{$[B]$}}
\put(07.00,23.00){\makebox(0,0)[cc]{$[A]$}}
\put(30.00,08.00){\makebox(0,0)[cc]{$[H]$}}
\end{picture}
\hspace*{\fill}
\end{minipage}%
\\[1ex]
The classes of the $E_i$ form a basis of $\Cl(S_1)$; 
the associated intersection matrix is
$$
{
\left(\begin{array}{@{}rrr@{\,}}
-1 & 1 & 1 \\ 1 & -1 & 0 \\ 1 & 0 & -1
\end{array}\right)}.
$$
This implies that $\,\id_{\Cl S_1}=
[H]\otimes [E_0] + [A]\otimes[E_1]+ [B]\otimes [E_2]$,
and the coefficients of $E_i$ in $\pDiv'_{\Cox S_1}$
are indeed $\Delta_{E_i}=\ko{0[E_i]}+\Nef(S_1)$.
\\[1ex]
For the second example $S_2$, we obtain the following pictures for the
fan and the class group, respectively. Again, $E_1$ is the exceptional
curve of the first blow-up, $E_2$ of the second blow-up, and $E_0$ is the
strict transform of the line.
\\[1ex]
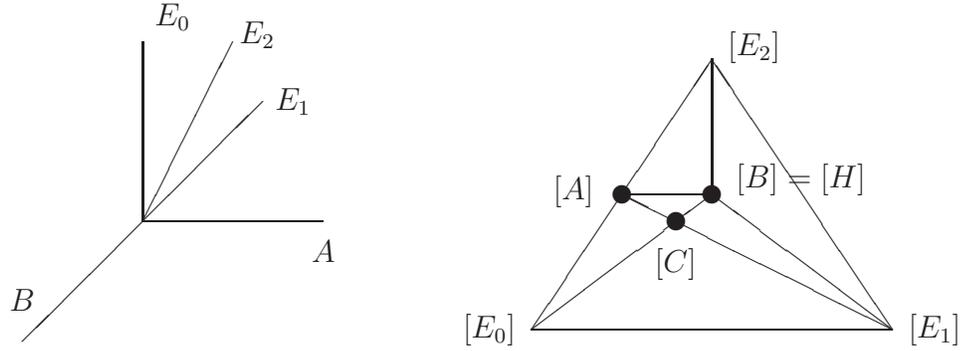
\begin{minipage}{15cm} 
\unitlength=0.8mm
\definecolor{cblowup}{rgb}{1.0,.388,.278}
\linethickness{0.4pt}
\hspace*{\fill}
\begin{picture}(70.00,70.00)(0,0)
\put(30.00,30.00){\line(1,0){30.00}}
\put(30.00,30.00){\line(1,1){20.00}}
\put(30.00,30.00){\line(1,2){15.00}}
\put(30.00,30.00){\line(0,1){30.00}}
\put(30.00,30.00){\line(-1,-1){20.00}}
\put(55.00,50.00){\makebox(0,0)[cc]{$E_1$}}
\put(49.00,61.00){\makebox(0,0)[cc]{$E_2$}}
\put(35.00,64.00){\makebox(0,0)[cc]{$E_0$}} 
\put(60.00,25.00){\makebox(0,0)[cc]{$A$}}
\put(10.00,17.00){\makebox(0,0)[cc]{$B$}}
\end{picture}
\hspace*{\fill}
\hspace*{\fill}
\begin{picture}(70.00,60.00)(0,-12)
\put(0.00,0.00){\line(1,0){60.00}}
\put(0.00,0.00){\line(2,3){30.00}} 
\put(60.00,0.00){\line(-2,3){30.00}} 
\put(0.00,0.00){\line(4,3){30.00}}
\put(60.00,0.00){\line(-4,3){30.00}}  %
\put(30.00,45.00){\line(0,-1){22.00}}  %
\put(60.00,0.00){\line(-2,1){45.00}}
\put(24.00,18.00){\circle*{3.00}}
\put(15.00,22.50){\circle*{3.00}}
\put(30.00,22.50){\circle*{3.00}}
\put(15.00,22.50){\line(1,0){15.00}}
\put(-7.00,00.00){\makebox(0,0)[cc]{$[E_0]$}}
\put(67.00,00.00){\makebox(0,0)[cc]{$[E_1]$}}
\put(37.00,47.00){\makebox(0,0)[cc]{$[E_2]$}}
\put(34.00,25.00){\makebox(0,0)[lc]{$[B]=[H]$}}
\put(07.00,23.00){\makebox(0,0)[cc]{$[A]$}}
\put(24.00,11.00){\makebox(0,0)[cc]{$[C]$}}
\end{picture}
\hspace*{\fill}
\end{minipage}%
\\[1ex]
Using the basis $\{[E_0], [E_1], [E_2]\}$, the intersection
matrix is as follows:
$$
{
\left(\begin{array}{@{}rrr@{\,}}
-1 & 1 & 0 \\ 1 & -2 & 1 \\ 0 & 1 & -1
\end{array}\right)}
$$
The pull back of the line is $[H]=[B]=[A]+[E_1]$
with $[A]=[E_2]+[E_0]$, $\,[B]=[E_0]+[E_1]+2[E_2]$, and
$[C]:=2[E_0]+[E_1]+2[E_2]=2[A]+[E_1]=[B]+[E_0]$
generating the nef cone $\Nef(S_2)$.
This implies that $\,\id_{\Cl S_2}=
[A]\otimes[E_1] + [C]\otimes [E_2] + [B]\otimes [E_0]$,
and the compact parts of the coefficients of the
$E_i$ in $\pDiv'_{\Cox S_2}$ are 
$\Delta^{\cp}_{E_0}=\ko{0[E_0]}$, but
$$
\textstyle
\Delta^{\cp}_{E_1}=\conv\{0,\;\frac{1}{2}[E_1],\; [E_1]+[E_2]\}
\hspace{0.7em}\mbox{and}\hspace{0.7em}
\Delta^{\cp}_{E_2}=\conv\{0,\;[E_2],\; [E_1]+2[E_2]\}.
$$

The two surfaces are homeomorphic; in fact, there exists a deformation
of $S_2$ to $S_1$.  Thus we can identify respective homology classes
and put them in one picture. The cohomology classes $[H]$, $[E_0]$,
$[E_2]$ and $[A]$ are the same for both surfaces, the class of the
second blow-up we denote by $[E_1]^1$ and $[E_1]^2$, respectively.  To
make the picture transparent the boundaries of $\Eff$ cones, as well as
their division in Zariski chambers, are denoted by dotted line
segments. 

$$
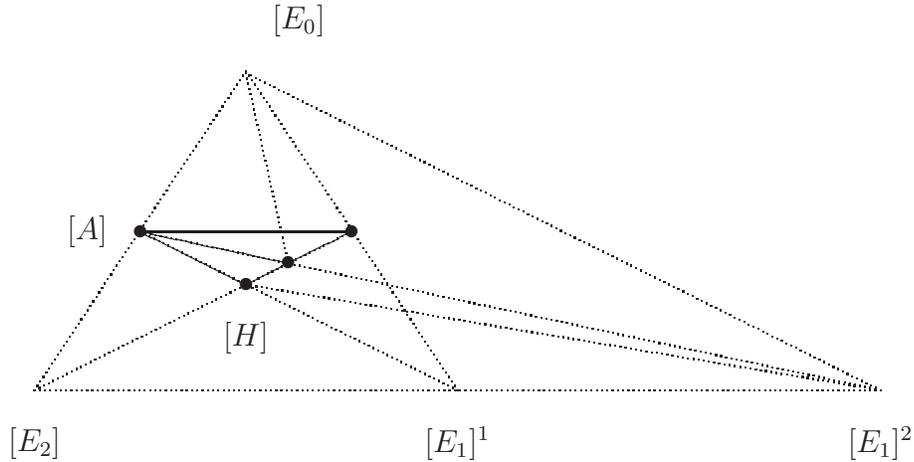
\begin{xy}<40pt,0pt>:
(0,-0.5)*={[E_2]}, (2.5,3.5)*={[E_0]},
(4,-0.5)*={[E_1]^1}, (8,-0.5)*={[E_1]^2},
(0.5,1.5)*={[A]}, (2,0.5)*={[H]}, 
(0,0)*={}="e2" ; (8,0)*={}="e12" **@{.},
(2,3)*={}="e0" ; (4,0)*={}="e11" **@{.},
"e0" ; "e12" **@{.},
"e0" ; "e2" **@{.},
(1,1.5)*={\bullet}="a" ; "e11" **@{.}, 
(3,1.5)*={\bullet}="b"; "e2" **@{.}, "b" ; "a" **@{-},
"a" ; "e12" **@{.},
(2,1)*={\bullet}="h" ; "e12" **@{.}, 
"h" ; "b" **@{-},
"h" ; "a" **@{-},
(2.4,1.2)*={\bullet}="c" ; "a" **@{-},
"c" ; "e0" **@{.}
\end{xy}
$$

This picture describes a typical situation: the effective cone, as the
function of a deformation is upper semicontinuous, that is
$\Eff(S_2)\supset\Eff(S_1)$ while the nef or movable cone is lower
semicontinuous, that is $\Mov(S_2)\subset\Mov(S_1)$.

\newcommand{\etalchar}[1]{$^{#1}$}

\end{document}